\documentclass[11pt]{article}
\usepackage{fullpage, amsmath}
\usepackage{amsthm,amsfonts,url,amssymb}
\usepackage{mathrsfs}
\usepackage{amssymb,amscd}
\usepackage{color}
\usepackage{chemarrow}
\usepackage{pictexwd,dcpic}
\usepackage{extarrows}
\usepackage[colorlinks,linkcolor=red,anchorcolor=green,citecolor=blue]{hyperref}
\usepackage{enumitem}
\usepackage{lipsum}
\usepackage[all,cmtip]{xy}
\usepackage{leftidx}
\usepackage[version=3]{mhchem}
\usepackage{imakeidx}

\newtheorem{theorem}{Theorem}[section]
\newtheorem{corollary}[theorem]{Corollary}
\newtheorem{lemma}[theorem]{Lemma}
\newtheorem{proposition}[theorem]{Proposition}

\newtheorem{remark}[theorem]{Remark}

\newtheorem{example}[theorem]{Example}

\newcommand{\hooklongrightarrow}{\lhook\joinrel\longrightarrow}
\newcommand{\twoheadlongrightarrow}{\relbar\joinrel\twoheadrightarrow}

\newcommand{\ra}{\rightarrow}
\newcommand{\lra}{\longrightarrow}

\newcommand{\Q}{\mathbb Q}

\newcommand{\Z}{\mathbb Z}

\newcommand{\bP}{\mathbb P}

\newcommand{\cR}{\mathcal R}

\newcommand{\cD}{\mathcal D}

\newcommand{\fh}{\mathfrak h}

\newcommand{\fp}{\mathfrak p}

\newcommand{\fz}{\mathfrak z}

\newcommand{\fc}{\mathfrak c}

\newcommand{\sL}{\mathscr L}

\newcommand{\sF}{\mathscr F}

\newcommand{\sG}{\mathscr G}

\DeclareMathOperator{\gl}{\mathfrak gl}

\DeclareMathOperator{\GL}{\mathrm GL}

\DeclareMathOperator{\Res}{\mathrm Res}

\DeclareMathOperator{\Sym}{\mathrm Sym}
\DeclareMathOperator{\Gal}{\mathrm Gal}
\DeclareMathOperator{\Hom}{\mathrm Hom}

\DeclareMathOperator{\End}{\mathrm End}

\DeclareMathOperator{\an}{\mathrm an}

\DeclareMathOperator{\dR}{\mathrm dR}

\DeclareMathOperator{\Ind}{\mathrm Ind}

\DeclareMathOperator{\Ima}{\mathrm Im}

\DeclareMathOperator{\id}{\mathrm id}

\DeclareMathOperator{\dett}{\mathrm det}

\DeclareMathOperator{\Ad}{\mathrm Ad}

\DeclareMathOperator{\wt}{\mathrm wt}

\DeclareMathOperator{\inte}{\mathrm int}
\DeclareMathOperator{\DF}{\mathrm DF}
\DeclareMathOperator{\pst}{\mathrm pst}
\DeclareMathOperator{\val}{\mathrm val}
\begin{document}

	\title{Change of weights for locally analytic representations of $\GL_2(\Q_p)$}
	\author{Yiwen Ding}
	\date{}
	\maketitle
		\begin{abstract}
		Let $D_1\subset D_2$ be $(\varphi, \Gamma)$-modules of rank $2$ over the Robba ring, and  $\pi(D_1)$, $\pi(D_2)$  be the associated locally analytic representations of $\GL_2(\Q_p)$ via the $p$-adic local Langlands correspondence. We describe the relation between $\pi(D_1)$ and $\pi(D_2)$. 
	\end{abstract}
\tableofcontents

	\section{Introduction and notation}
	Let $E$ be a finite extension of $\Q_p$, $\cR_E$ be the Robba ring of $\Q_p$ with $E$-coefficients. 
	Let $D$ be an indecomposable $(\varphi, \Gamma)$-module of rank $2$ over $\cR_E$. By \cite[Thm.~0.1]{Colm16}, the (locally analytic) $p$-adic Langlands correspondence associates to $D$ a locally analytic representation of $\GL_2(\Q_p)$ over $E$. 
	One phenomenon on the Galois side is that the $(\varphi, \Gamma)$-module $D$ has (infinitely) many $(\varphi, \Gamma)$-submodules $D'$, including trivial ones $\{t^i D\}_{i\in \Z_{\geq 0}}$ and some non-trivial ones discussed below. In this note, we describe the relation between $\pi(D)$ and $\pi(D')$. Note that the correspondence $D \mapsto \pi(D)$ is compatible with twisiting by characters. In particular, if  $D'=t^i D$, then  $\pi(D')\cong \pi(D) \otimes_E z^i \circ \dett$ (and we ignore $\dett$ when there is no ambiguity).

	Twisting by a certain character, we can and do assume $D$ has Sen weights $(0,\alpha)$ with $\alpha\in E \setminus \Z_{<0}$. For $k\in \Z_{\geq 1}$, denote by $V_k:=\Sym^{k} E^2$ the $k$-th symmetric product of the standard representation of $\GL_2(\Q_p)$. For a locally analytic representation $V$, we use $V^*$ to denote its strong continuous dual. Let $\fc\in \text{U}(\gl_2)$ be the Casimir operator.
\begin{theorem}\label{thmintro}
	(1) Assume $\End(D)\cong E$. Assume $\alpha\neq 0$, or $\alpha=0$ and $D$ not de Rham,  then $D$ admits a unique $(\varphi, \Gamma)$-submodule $D_{(0,\alpha+k)}$ of Sen weights $(0,\alpha+k)$ and we have ($[-]$ denoting the eigenspace)
	\begin{equation*}
		\pi(D_{(0,\alpha+k)})^* \cong (\pi(D)^* \otimes_E V_k)[\fc=(\alpha+k)^2-1].
	\end{equation*}

(2) Assume $\alpha=0$ and $D$ is de Rham non-trianguline, then $(\pi(D)^*\otimes_E V_k)[\fc=k^2-1]\cong \pi(D,k)^*$ and we have an exact sequence ($\{-\}$ denoting the generalized eigenspace)
\begin{equation*}
	0 \ra \pi(D,k)^* \ra (\pi(D)^* \otimes_E V_k)\{\fc=k^2-1\} \ra \pi(D,-k)^* \ra 0
\end{equation*}
	where $\pi(D,i)$ denotes Colmez's representations in \cite{Colm18} (for $D=\Delta$ of loc. cit.).
	\end{theorem}
\begin{remark}
(1) Some parts  of Theorem \ref{thmintro} (1) in trianguline case  were obtained in \cite[Thm.~5.2.11]{JLS21}.

(2)	A similar statement in Theorem (2) also holds in trianguline case, see Remark \ref{rem36} (3).

(3) Suppose we are in the case (2), and let $\pi_{\infty}(D)$ be the smooth representation of $\GL_2(\Q_p)$ associated to $D$ via the classical local Langlands correspondence. By \cite[Thm.~0.6 (iii)]{Colm18}, for any $(\varphi, \Gamma)$-submodule $D'$ of $D$ of Sen weights $(0,k)$, we have
\begin{equation*}
	0 \ra \pi(D')^* \ra \pi(D,k)^* \ra (\pi_{\infty}(D) \otimes_E V_k)^* \ra 0.
\end{equation*}
And the map $D' \ra \pi(D')^*$  gives a one-to-one correspondence between the  $(\varphi, \Gamma)$-submodules of $D$ of Sen weights $(0,k)$ and the subrepresentations of $\pi(D,k)^*$ of quotient $(\pi_{\infty}(D) \otimes_EV_k)^*$.

(4) Assume $D$ is not trianguline, the theorem allows to reconstruct Colmez's magical operator $\partial$ in \cite[Thm.~0.8]{Colm18} and generalize it to the general (irreducible) setting. Indeed, when $D$ is as in Theorem \ref{thmintro} (2), the composition (where the second map is induced by the map $V_k \ra E$, $\sum_{i=0}^k a_i e_1^i \otimes e_2^{k-i} \mapsto a_0$)
\begin{equation*}
(\pi(D,k)^* \otimes_E V_k)[\fc=k^2-1] \hooklongrightarrow \pi(D)^* \otimes_EV_k \twoheadlongrightarrow \pi(D)^*
\end{equation*}
is an isomorphism of topological vector spaces. The $\GL_2(\Q_p)$-action on $(\pi(D)^* \otimes_E V_k)[\fc=k^2-1]$ induces then a twisted $\GL_2(\Q_p)$-action on the space $\pi(D)^*$, and gives Colmez's formulas in the construction of $\pi(D,k)^*$. See \S~\ref{seccomp} for more details.
\end{remark}
Recall that a key ingredient in the construction of $\pi(D)$ is a delicate involution $w_D$ on $D^{\psi=0}$. When $D$ is irreducible, $w_D$ was obtained by continuously extending an involution on $(D^{\inte})^{\psi=0}$, where $D^{\inte}$ is the associated $(\varphi, \Gamma)$-module over $\cR_E^{\inte}:=B_{\Q_p}^{\dagger} \otimes_{\Q_p} E$. Let $D'\subset D$ be a submodule of weight $(0,\alpha+k)$. Then $\nabla_k:=(\nabla-k+1) \cdots (\nabla-1)\nabla D' \subset t^kD$, and we denote by $\frac{\nabla_k}{t^k}: D' \ra D$ the map sending $x$ to $t^{-k} \nabla_k(x)$. The involutions $w_D$ and $w_{D'}$ have the following simple relation (though they are in general not comparable when restricted to $D^{\inte}$ and $(D')^{\inte}$):
\begin{corollary}
	We have $w_{D'}=w_D \circ \frac{\nabla_k}{t^k}$.
\end{corollary}
We give a sketch of the proof of Theorem \ref{thmintro}. The key ingredient is an operation, that we call \textit{translation}, on $(\varphi, \Gamma)$-modules. The $\begin{pmatrix} \Z_p\setminus \{0\} & \Z_p \\ 0 & 1\end{pmatrix}$-action on $V_k$ induces an $\cR_E^+$-module structure on $V_k$ together with a semi-linear $(\varphi, \Gamma)$-action. In fact, we have $V_k\cong \cR_E^+/X^{k+1}$. For a $(\varphi, \Gamma)$-module $D$ over $\cR_E$, consider $D \otimes_E V_k$, equipped with the diagonal $\cR_E^+$-action and $(\varphi, \Gamma)$-action (noting $\cR_E^+$ has a natural coalgebra structure). One shows that the $\cR_E^+$-action on $D \otimes_E V_k$ uniquely extends to an $\cR_E$-action. In particular $D \otimes_EV_k$ is also a $(\varphi, \Gamma)$-module over $\cR_E$.

 Now let $D$ be as in Theorem \ref{thmintro}, then $D$ is naturally equipped with a $\gl_2$-action. We equip $D \otimes_EV_k$  with a diagonal $\gl_2$-action. The Casimir $\fc$ turns out to be an endomorphism of $(\varphi, \Gamma)$-modules of  $D \otimes_E V_k$. In particular, we can decompose $D\otimes_E V_k$ into generalized eigenspaces of $\fc$, which are $(\varphi, \Gamma)$-submodules over $\cR_E$. We study the decomposition in  \S~\ref{sec20}. For example, we show that if $D$ is as in Theorem \ref{thmintro} (1),  then  $D_{(0,\alpha+k)}\cong (D \otimes_EV_k)[\fc=(\alpha+k)^2-1]$. 
 
By \cite[Thm.~0.1]{Colm16},  there is a unique $\GL_2(\Q_p)$-sheaf over $\bP^1(\Q_p)$ of central character $\delta_D$ (which satisfies $\delta_D\varepsilon=\wedge^2 D$, $\varepsilon$ being the cyclotomic character), associated to $D$, whose global sections  $D \boxtimes_{\delta_D} \bP^1(\Q_p)$ sit in an exact sequence
 \begin{equation}\label{intro0}
 0 \ra \pi(D)^* \otimes_E \delta_D \ra D \boxtimes_{\delta_D} \bP^1(\Q_p) \ra \pi(D) \ra 0.
 \end{equation}
 It turns out that this construction is quite compatible with translations. 
Namely, to $D\otimes_E V_k$, one can naturally associate a $\GL_2(\Q_p)$-sheaf over $\bP^1(\Q_p)$ (of central character $\delta_D z^k$) whose global sections $(D\otimes_E V_k) \boxtimes_{\delta_Dz^k} \bP^1(\Q_p)$ are $\GL_2(\Q_p)$-equivariantly isomorphic to $(D \boxtimes_{\delta_D} \bP^1(\Q_p)) \otimes_E V_k$. Suppose $D$ is as in Theorem \ref{thmintro} (1), we then have (noting $\delta_D z^k=\delta_{D_{(0,\alpha+k)}}$)
\begin{multline}\label{intro1}
((D \boxtimes_{\delta_D} \bP^1(\Q_p)) \otimes_E V_k)[\fc=(\alpha+k)^2-1] \cong ((D\otimes_E V_k)[\fc=(\alpha+k)^2-1]) \boxtimes_{\delta_Dz^k} \bP^1(\Q_p)\\
\cong D_{(0,\alpha+k)} \boxtimes_{\delta_{D_{(0,\alpha+k)}}} \bP^1(\Q_p).
\end{multline}
Using the isomorphism and  (\ref{intro0}), one can deduce Theorem \ref{thmintro} (1). Theorem \ref{thmintro} (2) follows by similar arguments. A main difference is that in this case, the translations can only produce  $(\varphi, \Gamma)$-submodules  $t^i D$. For example, $(D \otimes_EV_k)[\fc=k^2-1]\cong D$ (noting the $\gl_2$-actions are however different), and $(D\otimes_E V_k)\{\fc=k^2-1\}\cong D \oplus t^k D$ (again, just as $(\varphi, \Gamma)$-module). 
 Similarly as in (\ref{intro1}), we deduce
\begin{equation*}
((D \boxtimes_{\delta_D} \bP^1(\Q_p)) \otimes_E V_k)[\fc=k^2-1] \cong D \boxtimes_{\delta_Dz^k} \bP^1(\Q_p)
\end{equation*}
and an exact sequence
\begin{equation*}
	0 \ra D \boxtimes_{\delta_Dz^k} \bP^1(\Q_p) \ra ((D \boxtimes_{\delta_D} \bP^1(\Q_p)) \otimes_E V_k)\{\fc=k^2-1\} \ra t^kD \boxtimes_{\delta_Dz^k} \bP^1(\Q_p) \ra 0.
\end{equation*}
Theorem \ref{thmintro} (2) follows then from these  together with results in \cite{Colm18}. We refer to the context for details.

\subsection*{Notation}	
Let $\varepsilon$ be the cyclotomic character of $\Gal_{\Q_p}$ and of $\Q_p^{\times}$.

\noindent We use the following notation for the Lie algebra $\gl_2$ of $\GL_2(\Q_p)$: $\fh:=\begin{pmatrix}
1 & 0 \\ 0 & -1
\end{pmatrix}$, $a^+:=\begin{pmatrix}
1 & 0 \\ 0 & 0
\end{pmatrix}$, $a^-:=\begin{pmatrix}
0 & 0 \\ 0 & 1 \end{pmatrix}$, $u^+:=\begin{pmatrix}
0 & 1 \\ 0 & 0
\end{pmatrix}$, $u^-:=\begin{pmatrix} 0 & 0 \\ 0 & 1
\end{pmatrix}$, $\fz:=a^++a^-$, and $\fc:=\fh^2-2\fh+4u^+u^-=\fh^2+2\fh+4u^-u^+\in \text{U}(\gl_2)$ be  the Casimir element.
	
\noindent Let $\cR_E$ be the $E$-coefficient Robba ring of $\Q_p$, and $\cR_E^+:=\{f=\sum_{n=0}^{+\infty} a_n X^n\ |\ f\in \cR_E\}$. Note $\cR_E^+$ is naturally isomorphic to the distribution algebra $\cD(\Z_p,E)$ on $\Z_p$. Let $t=\log(1+X)\in \cR_E^+\subset \cR_E$.  

\noindent We use $\bullet -- \bullet$ (resp. $\bullet-\bullet$) to denote a possibly split extension (resp. a non-split extension), with the left object the sub and the right object the quotient.

\subsection*{Acknowledgement}	I thank Zicheng Qian, Liang Xiao for helpful discussions. The work is supported by the NSFC Grant No. 8200800065 and No. 8200907289.

	\section{Translations of 
		$(\varphi, \Gamma)$-modules}
 We discuss some properties of  translations on $(\varphi, \Gamma)$-modules. 
	\subsection{Generalities}
	Let $k\in \Z_{\geq 0}$, let $V_k:=\Sym^k E^2$ be the algebraic representation of $\GL_2(\Q_p)$ of highest weight $(0,k)$ (with respect to $B(\Q_p)$). On $V_k$, we have $\fz=k$ and $\fc=k(k+2)$. The $P^+:=\begin{pmatrix} \Z_p\setminus \{0\} & \Z_p \\ 0 & 1\end{pmatrix}$-action induces   an $\cR_E^+$-structure on $V_k$ together with a semi-linear $(\varphi, \Gamma)$-action given by $(1+X) v=\begin{pmatrix} 1 & 1 \\ 0 & 1 \end{pmatrix} v$, $\varphi(v)=\begin{pmatrix} p & 0 \\ 0 & 1 \end{pmatrix}v$, $\gamma(v)=\begin{pmatrix}
		 \gamma & 0 \\ 0 & 1 
	\end{pmatrix} v$. Let $e$ be the lowest weight vector of $V_k$, then we have a $(\varphi, \Gamma)$-equivariant isomorphism of $\cR_E^+$-modules: $ \cR_E^+/X^{k+1} \xrightarrow{\sim} V_k$, $\alpha \mapsto \alpha e$. 
	
Let $D$ be a $(\varphi, \Gamma)$-module over $\cR_E$. Consider $D \otimes_E V_k$. We equip $D \otimes_E V_k$ with a diagonal $\cR_E^+$-action (using the coalgebra structure  $\cR_E^+ \ra \cR_E^+ \otimes_E \cR_E^+$, $(1+X) \mapsto (1+X) \otimes (1+X)$), and with a diagonal $(\varphi, \Gamma)$-action. It is clear that the resulting $(\varphi, \Gamma)$-action is $\cR_E^+$-semi-linear. We also equip $D \otimes_E V_k$ with the natural topology so that $D\otimes_E V_k \cong D^{\oplus k+1}$ as topological $E$-vector space. 	
	
\begin{proposition}
	The $\cR_E^+$-action on $D\otimes_E V_k$ uniquely extends to a continuous $\cR_E$-action. With this action,  $D\otimes_E V_k$ is a  $(\varphi, \Gamma)$-module over $\cR_E$ and is isomorphic to a successive extension: $t^k D -- t^{k-1} D -- \cdots -- tD -- D$.
\end{proposition}	
\begin{proof}
	We first prove the proposition for the case $k=1$. Let $e_0$ be the lowest weight vector in $V_1$, and $e_1:=X e_0$ (so $X e_1=0$). For $v=v_0 \otimes e_0 +v_1 \otimes e_1$, $X v=  X v_0 \otimes e_0+(v_0+X v_0+X v_1) \otimes e_1$. It is clear that $X$ is invertible on $D \otimes_EV_1$, and $X^{-1}(v_0 \otimes e_0+v_1 \otimes e_1)= (X^{-1} v_0) \otimes e_0+(-X^{-1}v_0-X^{-2}v_0+X^{-1}v_1) \otimes e_1$. For $f(T)\in E[T]$, $f(X^{-1})(v_0\otimes e_0+v_1 \otimes e_1)=(f(X^{-1}) v_0) \otimes e_0+(-(X^{-2}+X^{-1})f'(X^{-1}) v_0+f(X^{-1}) v_1)\otimes e_1 $. As $\cR_E$ acts on $D$, by the formula we see $\cR_E^+[1/X]$-action uniquely extends to a continuous $\cR_E$-action. The proposition in this case follows. Using induction, we see the $\cR_E^+$-action on $D \otimes_EV_1^{\otimes k}$ uniquely extends to a continuous $\cR_E$-action. As $D \otimes_E V_k$ is a (closed) direct summand of $D \otimes_E V_1^{\otimes k}$ stable by  $\cR_E^+$, it is also stabilized by $\cR_E^+[1/X]$ hence by $\cR_E$.
	
	Each $D\otimes_E (X^i \cR_E^+/X^{k+1})$, for $i=0,\cdots, k-1$, is clearly a $(\varphi, \Gamma)$-equivariant $\cR_E^+$-submodule. Using induction and the easy fact that $X$ is invertible on the graded pieces $D \otimes_E(X^j \cR_E^+/X^{j+1})$ (noting that the $\cR_E^+$-action on the graded pieces is the same as acting only on $D$), one easily sees that $X$ is invertible on $D\otimes_E (X^i \cR_E^+/X^{k+1})$. Hence $D\otimes_E (X^i \cR_E^+/X^{k+1})$ is a $(\varphi, \Gamma)$-submodule of $D \otimes_EV_k$ over $\cR_E$. On the graded piece, the induced $\cR_E$-action is the unique one that extends the $\cR_E^+$-action, hence coincides with the $\cR_E$-action on $D$. We then easily see that $D \otimes_E (X^i \cR_E^+/X^{i+1})$ isomorphic to $t^i D$. This concludes the proof. 
\end{proof}
\begin{remark} In particular, we have the following morphisms of $(\varphi, \Gamma)$-modules over $\cR_E$:
	\begin{eqnarray}
		D \otimes_E V_k \twoheadlongrightarrow D, \ &&  \sum_{i=0}^{k} v_i \otimes t^i e\mapsto v_0, \label{projna} \\
	D\hooklongrightarrow D \otimes_E V_k, \  && v \mapsto v \otimes t^k e.  \label{injna}
\end{eqnarray}
\end{remark}
\begin{example}\label{ex1}
We have	$\cR_E \otimes_E V_1\cong \cR_E \oplus t \cR_E$. Indeed, the element $1 \otimes e \in H^0_{(\varphi, \Gamma)}(\cR_E \otimes_E V_1)$. This induces a morphism $\cR_E \hookrightarrow \cR_E \otimes_E V_1$, whose composition with (\ref{projna}) (for $D=\cR_E$) is clearly an isomorphism. We see the extension in the proposition for $D=\cR_E$ and $k=1$ splits. See Remark \ref{remind} (1) for a non-split case.
\end{example}

\begin{remark}
Suppose $D$ is de Rham, then $D \otimes_E V_k$ is also de Rham. This easily follows from Proposition \ref{decork1} (1) below (which is obtained by using certain $\gl_2$-action). One can also directly prove it as follows. Indeed, by induction, it is sufficient to show $D \otimes_E V_1$ is de Rham. Let $\Delta$ be the $p$-adic differential equation associated to $D$ (of constant Hodge-Tate weight $0$), and $n\in \Z_{\geq 0}$ such that $t^n \Delta \subset D$. We see $t^n \Delta \otimes_E V_1$ is a $(\varphi, \Gamma)$-submodule of $D \otimes_E V_1$, and the both have the same rank. It suffices to show $\Delta \otimes_E V_1$ is de Rham. But we have (e.g. by \cite[Lem.~1.11]{Ding4}) $H^1_g(t \Delta \otimes_{\cR_E} \Delta^{\vee}) \xrightarrow{\sim} H^1(t \Delta \otimes_{\cR_E} \Delta^{\vee})$, hence any extension of $\Delta$ by $t \Delta$ is de Rham. 
\end{remark}
\begin{lemma}\label{psi}
For $v\otimes w\in D \otimes_E V_k$, we have $\psi(v \otimes w)=\psi(v) \otimes \varphi^{-1}(w)$
\end{lemma}
\begin{proof} Write $v=\sum_{i=0}^{p-1} (1+X)^i \varphi(v_i)$ (so $\psi(v)=v_0$). We have (using $\varphi$ is invertible on $V_k$): $$v\otimes w=\sum_{i=0}^{p-1}(1+X)^i (\varphi(v_i) \otimes (1+X)^{-i} w)=\sum_{i=0}^{p-1}(1+X)^i \varphi(v_i \otimes \varphi^{-1}((1+X)^{-i}w)).$$
The lemma follows.
\end{proof}
A $(\varphi, \Gamma)$-module $D$ is naturally equipped with a locally analytic action of  $P^+$, where $\begin{pmatrix}
1 & 1 \\ 0 & 1
\end{pmatrix}$ acts via $(1+X)\in \cR_E$, $\begin{pmatrix}
p & 0 \\ 0 & 1
\end{pmatrix}$ via $\varphi$, and $\begin{pmatrix}
\Z_p^{\times} & 0 \\ 0 & 1
\end{pmatrix}$ via $\Gamma$. Moreover, by  \cite[\S~1.3]{Colm16}, $D$ corresponds to a $P^+$-sheaf $\sF_{D}$ of analytic type over $\Z_p$, with the sections $\sF_{D}(i+p^n\Z_p)$ over $i+p^n \Z_p$, which we also denote by $D\boxtimes (i+p^n\Z_p)$ as in \textit{loc. cit.},  given by
\begin{equation}\label{res1}
	(1+X)^i \varphi^n \psi^n ((1+X)^{-i} v)=:\Res_{i+p^n\Z_p}(v)
\end{equation}
for $v\in D$. In particular, we have $D\boxtimes \Z_p^{\times} \cong D^{\psi=0}$.

 For a $P^+$-sheaf $\sF$ of analytic type over $\Z_p$, it is direct to check the following data defines a $P^+$-sheaf of analytic type $\sF \otimes_E V_k$ over $\Z_p$:
\begin{itemize}
	\item $(\sF \otimes_E V_k)(U):=\sF(U) \otimes_E V_k$,
	\item $\Res^V_U|_{\sF\otimes_E V_k} := \Res^V_U|_{\sF} \otimes \id$,
	\item $g_U|_{(\sF \otimes_E V_k)(U)}:=g_U|_{\sF(U)} \otimes g: (\sF \otimes_E V_k)(U) \ra (\sF \otimes_E V_k)(g(U))\cong \sF(g(U))\otimes_E V_k$ for $g\in P^+$.
\end{itemize}
\begin{lemma}\label{transP+}The identity map on $D \otimes_E V_k$ induces a natural $P^+$-equivariant isomorphism $\sF_{D} \otimes_E V_k \cong \sF_{D \otimes_E V_k}$ .
\end{lemma}
\begin{proof}
Let $i\in \Z_p$ and $n\in \Z_{\geq 0}$. For $x\otimes w\in D\otimes_E V_k$, using Lemma \ref{psi} and the formula in (\ref{res1}), we have 
	\begin{equation*}
		\Res_{i+p^n\Z_p}(x \otimes w)=\Res_{i+p^n \Z_p}(x) \otimes w.
	\end{equation*}
The identity map induces then an isomorphism of sheaves on $\Z_p$: $\sF_D \otimes_E V_k \cong \sF_{D \otimes_E V_k}$. It is straightforward to check the isomorphism is $P^+$-equivariant, as the both are equipped with the diagonal $P^+$-action.
\end{proof}
The following lemma is a direct consequence of Lemma \ref{transP+}. 
\begin{lemma}
	We have $(D \otimes_E V_k)^{\psi=0}=D^{\psi=0} \otimes_E V_k$ (as subspace of $D \otimes_E V_k$). Moreover, $\Res_{\Z_p^*}|_{D \otimes_E V_k}=\Res_{\Z_p^*}|_D \otimes \id$. 
\end{lemma}

\subsection{Translations of $(\varphi, \Gamma)$-modules of rank $2$}\label{sec20}
For a $(\varphi, \Gamma)$-module $D$ with an extra $\gl_2$-action, we study the  $(\varphi, \Gamma)$-module structure together with the (diagonal) $\gl_2$-action of the translation $D\otimes_E V_k$.

\subsubsection{A digression on $(\varphi, \Gamma)$-submodules} Let $D$ be a $(\varphi, \Gamma)$-module of rank $2$ over $\cR_E$. We briefly discuss the $(\varphi, \Gamma)$-submodules of $D$ and introduce some notation. Twisting $D$ by a rank one $(\varphi, \Gamma)$-module, we can and do assume that the Sen weights of $D$ are given by $0$ and $\alpha\in E\setminus \Z_{<0}$.  Let $D'$ be a $(\varphi, \Gamma)$-submodule of $D$, by \cite[Prop.~4.1]{Liu07}, there exists $n$ such that $D'\supset t^n D$. We are led to study the torsion $(\varphi, \Gamma)$-module $D/t^nD$. 
\begin{lemma}
	(1)	If $\alpha\notin \Z$, then there exists a locally analytic character of $\Q_p^{\times}$ of weight $\alpha$ such that $D/t^nD \cong \cR_E/t^n \oplus \cR_E(\chi_{\alpha})/t^n$. 
	
	(2) If $\alpha \in \Z_{\geq 0}$ and $D$ is not de Rham, then $D/t^nD$ is isomorphic to a non-split extension of $\cR_E/t^n$ by
	$\cR_E(z^\alpha)/t^n$.
	
	(3) If $\alpha \in \Z_{\geq 0}$ and  $D$ is de Rham, then $D/t^n D \cong \cR_E(z^{\alpha})/t^n \oplus \cR_E/t^n$. 
\end{lemma}
\begin{proof}
The lemma follows from Fontaine's classification of $B_{\dR}$-representations \cite[Thm.~3.19]{Fo04},  and \cite[Lem.~5.1.1]{BD2}.
\end{proof}
The following two propositions follow easily from the lemma.
\begin{proposition}[Non-de Rham case]
	(1) If $\alpha\notin \Z$, for any $n_1, n_2 \in \Z_{\geq 0}$, there exists a unique $(\varphi, \Gamma)$-submodule of $D$ of Sen weights $(n_1, \alpha+n_2)$, denoted by $D_{(n_1,\alpha+n_2)}$.  Moreover, any $(\varphi, \Gamma)$-submodule of $D$ of rank $2$ has this form.
	
	(2) If $\alpha \in \Z_{\geq 0}$ and $D$ is not de Rham, for $n_1, n_2 \in \Z_{\geq 0}$, $n_1\leq n_2+\alpha$, there exists a unique $(\varphi, \Gamma)$-submodule of $D$ of Sen weights $(n_1, \alpha+n_2)$, denoted by $D_{(n_1, \alpha+n_2)}$. Moreover, any $(\varphi, \Gamma)$-submodule of $D$ of rank $2$ has this form.
\end{proposition}
\begin{remark}
	When $n_1=n_2=n$, then the $(\varphi, \Gamma)$-submodule of $D$ of weights $(n, \alpha+n)$ is just $t^n D$. 
\end{remark}
\begin{proposition}[De Rham case] Assume $D$ is de Rham.
	
	(1) For each $n \in \Z$, $n\geq \alpha$, there exists a unique $(\varphi, \Gamma)$-submodule of $D$ of Sen weights $(n,n)$.  
	
	(2) If $\alpha\in \Z_{\geq 1}$, for $n_1 \in \Z_{\geq 0}$, $n_1\leq \alpha$  (resp. $n_2\in \Z_{\geq 0}$, there exist a unique $(\varphi, \Gamma)$-submodule of $D$ of Sen weights $(n_1, \alpha)$ (resp. $(0, \alpha+n_2)$), which we denote by $D_{(n_1,\alpha)}$ (resp. $D_{(0,\alpha+n_2)}$). 
	
	(3) If $\alpha=0$, let $n\geq 1$, the $(\varphi,\Gamma)$-submodules of $D$ of Sen weights $(0,n)$ are parametrized by lines $\sL\subset D_{\dR}(D)$, each denoted by $D_{n,\sL}$.
\end{proposition}
\begin{remark}
(1)	By the proposition, one easily gets a full description of $(\varphi, \Gamma)$-submodules of $D$ in de Rham case. Note also that in case (3), $D_{n,\sL}$ can be isomorphic for different $\sL$. 

(2) Assume $D$ is de Rham and $\alpha\geq 1$. For $n_1, n_2\in \Z_{\geq 0}$, $n_1 \leq n_2+\alpha$, by (2) there exists a unique $(\varphi, \Gamma)$-submodule, denoted by $D_{(n_1, \alpha+n_2)}$, of $D$ of Sen weights $(n_1, \alpha+n_2)$ such that $D_{(n_1, \alpha+n_2)} \subset D_{(0,\alpha+n_2)}$. We have $D_{(n, \alpha+n)}\cong t^{n} D$. Recall there is an equivalence of categories between de Rham $(\varphi,\Gamma)$-modules and filtered Deligne-Fontaine modules (cf. \cite[Thm.~A]{Ber08a}). The associated Deligne-Fontaine modules (igonring the Hodge filtration) of $D_{(n_1, \alpha+n_2)}$ are all the same and isomorphic to that of $D$. For the Hodge filtration, to go from $D$ to $D_{(n_1, \alpha+n_2)}$ with $n_1<\alpha+n_2$, one just shifts the degree of the filtration respectively. 
\end{remark}

\subsubsection{Translations}\label{sec22} For a $(\varphi, \Gamma)$-module $D$ over $\cR_E$, by differentiating the $P^+$-action, we obtain an action of its Lie algebra $\fp^+$ on $D$, where $a^+$ acts via the operator $\nabla$ (given by differentiating the $\Gamma$-action), and $u^+$ acts via $t$. We  assume $D$ is of rank $2$ and  has Sen weights $(0,\alpha_D)$, with $\alpha_D\in E\setminus \Z_{<0}$. Let $P(T)\in E[T]$ be the monic Sen polynomial of $D$, hence $P(\nabla) (D)\subset tD$ (e.g. see \cite[Lem.~1.6]{Colm18}). For an operator $\nabla'$ such that $\nabla'(D) \subset tD$, we use $\frac{\nabla'}{t}$ to denote the operator mapping $x$ to $\frac{1}{t} \nabla'(x)$. In particular, we have the operator $\frac{P(\nabla)}{t}$ on $D$.  We recall the $\gl_2$-actions on $D$, and we refer to \cite[\S~3.2.1]{Colm16} for  details. We restrict to the case with infinitesimal character for our applications.
\begin{proposition}\label{lieaction}(1) If  $\deg P(T)=2$ (so is equal to $T(T-\alpha_D)$), then there exists a unique $\gl_2$-action on $D$ extending the $\fp^+$-action satisfying that $D$ has infinitesimal character. The action is given by $u^-=-\frac{P(\nabla)}{t}$ and $\fz=\alpha_D-1$. Consequently, $\fc$ acts via $\alpha_D^2-1$.

(2) If $\deg P=1$ (so $\alpha_D=0$ and $P(T)=T$), for $\alpha\in E$, there exists a unique $\gl_2$-action on $D$ extending the $\fp^+$-action satifying that $D$ has infinitesimal character and $\fz$ acts via $\alpha-1$. The action is given by $u^-=-\frac{\nabla(\nabla-\alpha)}{t}$. Consequently, $\fc=\alpha^2-1$.
\end{proposition}\label{rlieact}
\begin{remark}(1) By \cite[Prop.~3.4]{Colm16}, if  $D$ does not contain a pathological  $(\varphi, \Gamma)$-submodule (cf. \cite[Rem.~3.5]{Colm16}), then the uniqueness in the proposition already holds with the condition having infinitesimal character replaced by having central character (for $\fz$). If $D$ contains a pathological $(\varphi, \Gamma)$-submodule, then the $\fp^+$-action on $D$ can extend to a $\gl_2$-action without infinitesimal character (but with central character).
	
	(2) For a general rank two $(\varphi, \Gamma)$-module $D'$, there exist $D$ as above and a continuous character $\chi$ such that $D'\cong D \otimes_{\cR_E} \cR_E(\chi)$. The $\gl_2$-action on $D'$ is then given by twisting the one on $D$ by $d \chi \circ \dett$.
\end{remark}

In the sequel, we let $\alpha\in E$ such that $D$ is equipped with the $\gl_2$-action with $u^-=-\frac{\nabla(\nabla-\alpha)}{t}$, $\fz=\alpha-1$. For example, $\alpha=\alpha_D$ (hence $\alpha\notin \Z_{<0}$ in this case) if $\deg P(T)=2$. If $P(T)=T$, $\alpha$ can be arbitrary. We equip $D\otimes_E V_k$ with a natural (diagonal) $\gl_2$-action. Note that on $D\otimes_E V_k$, $\fz=\alpha+k-1$.
\begin{lemma}
	The Casimir operator $\fc$ on $D\otimes_E V_k$ defines an endomorphism of $(\varphi, \Gamma)$-modules over $\cR_E$. 
\end{lemma}
\begin{proof}The operator $\fc$ commutes with the adjoint action of $P^+$. Hence
 $\fc$ is $\cR_E^+$-linear and commutes with $\varphi$ and $\Gamma$. The lemma  follows from $\End_{\cR_E,(\varphi, \Gamma)}(D\otimes_E V_k) \xrightarrow{\sim} \End_{\cR_E^+, (\varphi, \Gamma)}(D\otimes_E V_k)$.
\end{proof}
By the lemma, we can decompose $D\otimes_EV_k$ into a direct sum of generalized eigenspaces $(D \otimes_E V_k)\{\fc=\mu\}$, each being a saturated $(\varphi, \Gamma)$-submodule of $D \otimes_E V_k$. 
\begin{lemma}\label{inj}
	For any $\mu\in E$, the composition $(D\otimes_E V_k)[\fc=\mu] \hookrightarrow D\otimes_EV_k \xrightarrow{(\ref{projna})} D$ is injective. 
\end{lemma}
\begin{proof}
	Let $e_{i}=t^ie$ for $i=0, \cdots, k$, and $v=\sum_{i=0}^{k} v_i \otimes e_i$. We have (letting $v_{-1}=v_{k+1}=0$)
	\begin{equation}\label{casmirk}
		\fc v= \sum_{i=0}^{k} (\fc v_i) \otimes e_i + \sum_{i=0}^{k+1} v_i \otimes (\fc e_i) + \sum_{i=0}^{k} (4u^- v_{i-1}+4(i+1)(k-i) u^+ v_{i+1}+2(2i-k) \fh v_i) \otimes e_i.
	\end{equation}
If $\fc v=\mu v$ and $v_{0}=0$, comparing the $e_0$ terms on both sides and using $u^+$ is injective on $D$, we easily see $v_1=0$. Using induction and  similar argument (comparing the $e_{i-1}$ term), we see $v_i=0$ for all $i$ hence $v=0$. The lemma follows.
\end{proof}
Now  consider the $k=1$ case.
\begin{lemma}\label{keylm}
	(1) Suppose $\alpha_D\notin \Z$ or $\alpha_D\in \Z_{\geq 1}$, then we have a $\gl_2(\Q_p)$-equivariant isomorphism of $(\varphi, \Gamma)$-modules over $\cR_E$: 
	\begin{equation*}
		D\otimes_E V_1 \cong (D\otimes_E V_1)[\fc=(\alpha+1)^2-1]  \oplus (D\otimes_E V_1)[\fc=(\alpha-1)^2-1]=D_{(0, \alpha+1)} \oplus D_{(1,\alpha)}.
	\end{equation*}

(2) Suppose $\alpha_D=0$ and $D$ is not de Rham, then $D\otimes_E V_1 =(D\otimes_E V_1)\{\fc=0\}$ 
which, as $(\varphi, \Gamma)$-module or $\gl_2$-module, is isomorphic to a non-split self-extension of $(D\otimes V_1)[\fc=0]\cong D_{(0,1)}$.

(3) Suppose $\alpha_D=0$ and $D$ is de Rham. Then $D\otimes_E V_1 \cong D \oplus t D$. And we have
\begin{enumerate}
	\item [(a)] If $\alpha\neq 0$, then $D\otimes_E V_1\cong (D\otimes_E V_1)[\fc=(\alpha+1)^2-1]  \oplus (D\otimes_E V_1)[\fc=(\alpha-1)^2-1]=D \oplus tD$.
	\item [(b)] If $\alpha=0$, then $D\otimes_E V_1 =(D\otimes_E V_1)\{\fc=0\}=(D\otimes_E V_1)[\fc^2=0]$, and $(D\otimes V_1)[\fc=0]\cong D$.
\end{enumerate}
\end{lemma}
\begin{proof} Let $e_0=e$ and $e_1=t e$.  For $v=v_0\otimes e_0+v_1 \otimes e_1\in D \otimes_E V_1$, by (\ref{casmirk})
	\begin{equation*}
		\fc (v) = (\alpha^2+2) v+(-2\fh v_0+4u^+ v_1) \otimes e_0+(4u^- v_0+2\fh v_1) \otimes e_1.
	\end{equation*}
Suppose $\fc(v)=(\alpha^2+2+\lambda) v$, then (noting $\fh=2\nabla-\alpha+1$):
\begin{equation*}
	\begin{cases}
		(4 \nabla-2\alpha+2+\lambda) v_0=4t v_1 \\
		(4 \nabla -2\alpha+2-\lambda) v_1=\frac{4\nabla(\nabla-\alpha)}{t} v_0
	\end{cases}.
\end{equation*}
As $\nabla(tx)=t(\nabla+1)x$ for $x\in D$, we deduce 
\begin{equation*}
	4\nabla(\nabla-\alpha) v_0=(4\nabla-2\alpha-2-\lambda) tv_1=\frac{1}{4}(4\nabla-2\alpha-2-\lambda)	(4 \nabla-2\alpha+2+\lambda) v_0.
\end{equation*}
If $v \neq 0$ (hence $v_0\neq 0$), $\lambda=\pm 2\alpha-2$. We also see
\begin{equation}\label{inf0}
	(D \otimes_E V_1)[\fc=(\alpha\pm 1)^2-1]=\big\{v_0 \otimes e_0+\frac{\nabla-\alpha/2\alpha\pm \alpha/2}{t}v_0 \otimes e_1  \ |\ v_0\in D, \ (\nabla-\alpha/2\pm \alpha/2)v_0 \subset tD\big\}.
\end{equation}
It is clear that  the image of $(D \otimes_E V_1)[\fc=(\alpha\pm 1)^2-1] \hookrightarrow D$ contains $tD$.  In particular, $(D \otimes_E V_1)[\fc=(\alpha\pm 1)^2-1]$ is a $(\varphi, \Gamma)$-submodule of $D$ of rank $2$. If $\alpha\neq 0$, we have then 
\begin{equation}\label{keydeo}
	(D \otimes_E V_1)[\fc=(\alpha+1)^2-1] \oplus 	(D \otimes_E V_1)[\fc=(\alpha-1)^2-1] \xrightarrow{\sim} 	D \otimes_E V_1.
\end{equation} 
If $\alpha=0$, we have $D \otimes_E V_1\cong (D \otimes_E V_1)\{\fc=0\}$. Moreover, by direct calculation, we have $(D \otimes_E V_1)\{\fc=0\}=(D \otimes_E V_1)[\fc^2=0]$.

Next we describe the $(\varphi, \Gamma)$-module  $(D \otimes_E V_1)[\fc=(\alpha\pm 1)^2-1]$. We will use the $\gl_2$-action (although one can directly describe it using (\ref{inf0})). 
For $x\in (D \otimes_E V_1)[\fc=(\alpha+1)^2-1]$, using $\fc=\fh^2-2\fh+4u^+u^-$, we have $\nabla (\nabla-\alpha-1) (x)\in t ((D \otimes_E V_1)[\fc=(\alpha+1)^2-1])$. Hence the Sen polynomial of $(D \otimes_E V_1)[\fc=(\alpha+1)^2-1]$ divides $T(T-\alpha-1)$. Similarly, the Sen polynomial of $(D \otimes_E V_1)[\fc=(\alpha-1)^2-1]$ divides $(T-1)(T-\alpha)$. 

 If $D$ is not de Rham, then $(D \otimes_E V_1)[\fc=(\alpha\pm 1)^2-1]$ twisted by any character is also not de Rham. Hence its Sen polynomial has to be of degree $2$. Together with the above discussion,  we easily deduce $(D \otimes_E V_1)[\fc=(\alpha+1)^2-1]=D_{(0,\alpha+1)}$, and if moreover $\alpha\neq 0$,  $(D \otimes_E V_1)[\fc=(\alpha-1)^2-1]\cong D_{(1,\alpha)}$. If $\alpha=0$, we see $(D\otimes_E V_1)/(D \otimes_E V_1)[\fc=0]$ is also a $(\varphi, \Gamma)$-module of rank $2$ of Sen weights $(0,1)$. Consider the composition 
 \begin{equation}\label{compo2}
 	t D \xrightarrow{(\ref{injna})} D \otimes_E V_1 \twoheadrightarrow (D\otimes_E V_1)/(D \otimes_E V_1)[\fc=0].
 \end{equation}
By Lemma \ref{inj},  $tD \cap (D \otimes_E V_1)[\fc=0]=0$ (as submodules of $D\otimes_E V_1$), hence (\ref{compo2}) is also injective. We deduce $(D\otimes_E V_1)/(D \otimes_E V_1)[\fc=0]\cong D_{(0,1)}$. If $D\otimes_E V_1\cong D_{(0,1)} \oplus D_{(0,1)}$, then $D \otimes_E V_1$ is Hodge-Tate, which is impossible as its saturated $(\varphi, \Gamma)$-submodule $tD$ is not Hodge-Tate. So $D\otimes_E V_1$ is a non-split self-extension of $D_{(0,1)}$. 
This finishes (2) (and (1) for non-de Rham case). 

Assume now $D$ is de Rham, and suppose first $D$ has distinct Sen weights. Then the $\gl_2$-action on $D$ is unique and $\alpha=\alpha_D \in \Z_{\geq 1}$. The Sen weights of  $(D \otimes_E V_1)[\fc=(\alpha+1)^2-1]$ are $(0,\alpha+1)$ or $(0,0)$ or $(\alpha+1,\alpha+1)$. However, $(D \otimes_E V_1)[\fc=(\alpha+1)^2-1]$ is a submodule of $D$ (resp. a saturated submodule of $D\otimes_EV_1$), so it can not have Sen weights $(0,0)$ (resp. $(\alpha+1,\alpha+1)$). We deduce hence $(D \otimes_E V_1)[\fc=(\alpha+1)^2-1]\cong D_{(0,\alpha+1)}$. As $(D \otimes_E V_1)[\fc=(\alpha-1)^2-1]$ has Sen weights $(1,\alpha)$ and contains $tD$, we see $(D \otimes_E V_1)[\fc=(\alpha-1)^2-1]\cong D_{(1,\alpha)}$. 

Finally, suppose $D$ is de Rham of weights $(0,0)$. By (\ref{inf0}) and $\nabla D \subset tD$, the composition $(D\otimes_E V_1)[\fc=(\alpha+1)^2-1] \hookrightarrow D \otimes_E V_1 \twoheadrightarrow D$ is surjective. We have thus $D\otimes_E V_1\cong D \oplus tD$.  By comparing the Sen weights, the injection $tD \rightarrow (D\otimes_E V_1)/(D\otimes_E V_1)[\fc=(\alpha+1)^2-1]$ (obtained as in (\ref{compo2})) is also an isomorphism. (3) follows.
\end{proof}
\begin{remark}\label{remind}
(1) Assume $\alpha_D\neq 0$, then we obtain two filtrations on $D\otimes_E V_1$:
\begin{equation*}
	D\otimes_E V_1 \cong D_{(0,\alpha+1)} \oplus D_{(1,\alpha)} \cong [tD--D].
\end{equation*}
If $D$ is not trianguline, then $\Hom_{(\varphi, \Gamma)}(D, D_{(0,\alpha+1)})=\Hom_{(\varphi, \Gamma)}(D,D_{(1,\alpha)})=0$, so the extension $tD--D$ is non-split. 

(2) The induced $\gl_2$-action on $(D\otimes_EV_1)[\fc=(\alpha\pm 1)^2-1]$, and $(D\otimes_E V_1)\{\fc=0\}/(D\otimes_E V_1)[\fc=0]$ (when $\alpha=0$) coincides with the one in Proposition \ref{lieaction} and Remark \ref{rlieact} (2).

\end{remark}
\begin{proposition}\label{transProp}
	(1) Suppose $\alpha_D\notin \Z$ or $\alpha_D \in \Z$ and $\alpha_D \geq k$, then we have a $\gl_2$-equivariant isomorphism of $(\varphi, \Gamma)$-modules over $\cR_E$: 
	\begin{equation*}
		D\otimes_E V_k \cong \oplus_{i=0}^k(D\otimes_E V_k)[\fc=(\alpha+k-2i)^2-1]=\oplus_{i=0}^k D_{(i, \alpha+k-i)}.
	\end{equation*}

(2) Suppose $\alpha_D \neq 0$, then $(D\otimes_E V_k)\{\fc=(\alpha+k)^2-1\}=(D\otimes_E V_k)[\fc=(\alpha+k)^2-1]\cong D_{(0,\alpha+k)}$.

(3) Suppose $\alpha_D=0$ and $D$ is not de Rham, $(D\otimes_E V_k)\{\fc=k^2-1\}$ is a non-split self-extension of $(D\otimes_E V_k)[\fc=k^2-1]\cong D_{(0,k)}$. 

(4) Suppose $\alpha_D=0$ and  $D$ is de Rham. If $\alpha \in E\setminus \Z_{<0}$, then $(D\otimes_E V_k)[\fc=(\alpha+k)^2-1]\cong D$, and if $\alpha=0$,  $(D\otimes_E V_k)\{\fc=k^2-1\}=(D\otimes_E V_k)[(\fc-k^2+1)^2=0]\cong D \oplus t^k D$. 
\end{proposition}
\begin{proof}
	By Lemma \ref{keylm} (and the proof) and an easy induction argument (see also Remark \ref{remind} (2)), we have
	\begin{equation}\label{decomk}
		D \otimes_EV_1^{\otimes k} \cong \sum_{i=0}^k (D\otimes_E V_1^{\otimes k})\{\fc=(\alpha+k-2i)^2-1\}.
	\end{equation}
And if $\alpha\notin \Z$ or $\alpha\in \Z$, $\alpha \geq k$, by Lemma \ref{keylm} (1) and induction, we have (where the factors in the direct sum can have rank bigger than $2$)
	\begin{equation}\label{decom2}
	D \otimes_EV_1^{\otimes k} \cong \oplus_{i=0}^k (D\otimes_E V_1^{\otimes k})[\fc=(\alpha+k-2i)^2-1].
\end{equation}
The first isomorphism in (1) follows. By Lemma \ref{casmirk}, $(D\otimes_E V_k)[\fc=(\alpha+k-2i)^2-1]$ is a $(\varphi, \Gamma)$-submodule of $D$ of rank  at most $2$.  By similar arguments as in the proof of Lemma \ref{keylm}, the Sen polynomial of $(D\otimes_E V_k)[\fc=(\alpha+k-2i)^2-1]$ divides $(X-i)(X-(\alpha+k-i))$. As the Sen weights of $D\otimes_E V_k$ are given by $(0, \cdots, k, \alpha, \cdots, \alpha+k)$, by comparing the weights, we see the Sen weights of $(D\otimes_E V_k)[\fc=(\alpha+k-2i)^2-1]$ are exactly $(i, \alpha+k-i)$. It rests to show 
\begin{equation}\label{compar0}(D\otimes_E V_k)[\fc=(\alpha+k-2i)^2-1]\cong D_{(i,\alpha+k-i)}.
\end{equation} The $k=1$ case was proved in Lemma \ref{keylm} (1). Assume $k\geq 2$. We use induction and assume hence  (1) holds for $k'<k$.  For $i=0$ (resp. $i=k$), (\ref{compar0}) holds as $D_{(0,\alpha+k)}$ (resp. $D_{(k,\alpha)}$) is the unique submodule of $D$ of Sen weights $(0, \alpha+k)$ (resp. $(k, \alpha)$).  Assume $1 \leq i \leq k-1$, it  is easy to see $(D \otimes_E V_k)[\fc=(\alpha+k-2i)^2-1]$ is a direct summand of $((D\otimes_E V_{k-1}) \otimes_E V_1)[\fc=(\alpha+k-2i)^2-1]$. As (1) holds fo $k-1$, it is not difficult to see the latter is isomorphic to 
\begin{equation}\label{ind11}
	(D_{(i,\alpha+k-1-i)} \otimes_E V_1)[\fc=(\alpha+k-2i)^2-1] \oplus (D_{(i-1,\alpha+k-i)} \otimes_E V_1)[\fc=(\alpha+k-2i)^2-1]\cong D_{(i,\alpha+k-i)}^{\oplus 2}.
\end{equation}
By Clebsh-Gordan rule,  we have $V_{k-2} \otimes_E (\wedge^2 V_1) \hookrightarrow V_{k-1} \otimes_EV_1 \twoheadrightarrow V_k$ and the composition is zero. By the induction hypothesis for $k-2$, we have $$(D \otimes_E V_{k-2} \otimes_E (\wedge^2 V_1))[\fc=(\alpha+k-2i)^2-1]\cong t D_{(i-1, \alpha+k-1-i)}\cong D_{(i,\alpha+k-i)}.$$
Together with (\ref{ind11}) and the fact $(D \otimes_E V_k)[\fc=(\alpha+k-2i)^2-1]$ has Sen weights $(i,\alpha+k-i)$, it is not difficult to deduce (\ref{compar0}). This finishes the proof of (1).

(2) By (\ref{decomk}), one can easily show that $(D\otimes_E V_1^{\otimes i})\{\fc=(\alpha+k)^2-1\}= 0$ for $i<k$ (note in this case $\alpha\notin \Z_{<0}$). Using Lemma \ref{keylm} (1) and induction, we also have  $(D\otimes_E V_1^{\otimes k})\{\fc=(\alpha+k)^2-1\}=(D\otimes_EV_1^{\otimes k})[\fc=(\alpha+k)^2-1]\cong D_{(0,\alpha+k)}$. (2) follows.

(3) By (\ref{decomk}),  $(D \otimes_E V_1^{\otimes i})\{\fc=k^2-1\}=0$ for $i<k$. It suffices to show the same statement with $V_k$ replaced by $V_1^{\otimes k}$. By Lemma \ref{inj} and an induction argument using Lemma \ref{keylm} (1), we get $(D\otimes_E V_1^{\otimes k})[\fc=k^2-1]\cong D_{(0,k)}$.
By Lemma \ref{keylm} (2) and induction,  it is not difficult to see $(D \otimes_E V_1^{\otimes k})\{\fc=k^2-1\}$ is a self-extension of $(D\otimes_E V_1^{\otimes k})[\fc=k^2-1]\cong D_{(0,k)}$. We see the statement in (3) except the non-split property holds.  If the extension splits, the multiplicity of $(T-k)$ in the Sen polynomial of $D \otimes_E V_k$ is one (noting $k$ is not a Sen weight of $(D\otimes_E V_k)/ ((D\otimes_EV_k)\{\fc=k^2-1\})$ by (\ref{decomk}) and the discussion in the first paragraph), however the saturated $(\varphi, \Gamma)$-submodule $t^k D$ of $D \otimes_E V_k$ is not Hodge-Tate, a contradiction.

(4) Again by (\ref{decomk}), if $\alpha\notin \Z_{<0}$,  $(D\otimes_E V_1^{ \otimes i})\{\fc=(\alpha+k)^2-1\}= 0$ for $i<k$, hence it suffices to prove the same statement for $V_1^{\otimes k}$. By Lemma \ref{inj} and  Lemma \ref{keylm} (3) with an induction argument, we have $(D\otimes_E V_1^{\otimes k})[\fc=(\alpha+k)^2-1]\cong D$. Assume now $\alpha=0$, by Lemma \ref{keylm} (3), we have an exact sequence (which splits as $(\varphi, \Gamma)$-module)
\begin{equation*}
	0 \ra (D \otimes_E V_1^{\otimes (k-1)})\{\fc=k^2-1\}  \ra (D\otimes_E V_1^{\otimes k})\{\fc=k^2-1\} \ra (tD \otimes_E V_1^{\otimes (k-1)})\{\fc=k^2-1\} \ra 0,
\end{equation*}
where the $\gl_2$-action on $D$ in the left term (resp. on $tD$ in the right term) fits into Lemma \ref{keylm} (3)(a) for $\alpha=1$ (resp. for $\alpha=-1$, after an appropriate twist).  By (\ref{decomk}) and an induction argument using (\ref{keydeo}), we have 
$(D \otimes_E V_1^{\otimes (k-1)})\{\fc=k^2-1\} \cong (D \otimes_E V_1^{\otimes (k-1)})[\fc=k^2-1]\cong D$, and  $(tD \otimes_E V_1^{\otimes (k-1)})\{\fc=k^2-1\}\cong (tD \otimes_E V_1^{\otimes (k-1)})[\fc=(-k)^2-1]$ is a $(\varphi, \Gamma)$-submodule of rank $2$ of $tD$, and has Sen polynomial dividing $T(T-k)$ (by similar arguments as in the proof of Lemma \ref{keylm}). As $(D\otimes_E V_1^{\otimes k})\{\fc=k^2-1\} \cong (D\otimes_E V_k)\{\fc=k^2-1\}$ is saturated in $D \otimes_E V_k$, we easily deduce $(tD \otimes_E V_1^{\otimes (k-1)})[\fc=k^2-1]$ has constant Sen weight $k$, hence is isomorphic to $t^k D$. This concludes the proof. 
\end{proof}
\begin{remark}\label{rem221}
We will frequently use the following special case:	
\begin{equation}\label{wtk}(D\otimes_E V_k)[\fc=(\alpha+k)^2-1] \cong \begin{cases}
		D & \alpha=0 \text{ and $D$ is de Rham}  \\ D_{(0, \alpha+k)} &\text{ otherwise}
	\end{cases}.
\end{equation}
Note that by induction, we also have
\begin{multline}\label{ind2}
(D\otimes_E V_k)[\fc=(\alpha+k)^2-1]\cong (D\otimes_E V_1^{\otimes k})[\fc=(\alpha+k)^2-1]\\
\cong \big(( D\otimes_E V_1^{\otimes (k-1)})[\fc=(\alpha+k-1)^2-1]\otimes_EV_1\big)[\fc=(\alpha+k)^2-1].
\end{multline}
From this and (\ref{inf0}), we have the following uniform description of $(D\otimes_E V_k)[\fc=(\alpha+k)^2-1]$ as submodule of $D$ (which can also be directly deduced from (\ref{wtk}))
\begin{equation*}
(D\otimes_E V_k)[\fc=(\alpha+k)^2-1]=\{x\in D\ |\ \nabla_i(x)\in t^i D, \ \forall i=1, \cdots, k\}
\end{equation*}
where $\nabla_i:=(\nabla-i+1)\cdots (\nabla-1) \nabla$. 
\end{remark}

Finally we quickly discuss the translation on general $p$-adic differential equations, where everything is essentially the same as the rank two case. Let $\Delta$ be a de Rham  $(\varphi, \Gamma)$-module of  constant Hodge-Tate weight $0$. For $\alpha\in E$, by \cite[Prop.~3.6]{Colm16}, we equip  $\nabla$ with a $\gl_2$-action extending the natural $\fp^+$-action such that $\fz=\alpha-1$ and $u^-=-\frac{\nabla(\nabla-\alpha)}{t}$ (so $\fc=\alpha^2-1$).  Note that Lemma \ref{inj} still holds with $D$ replaced by $\nabla$. 
\begin{proposition}\label{decork1}
(1) $\Delta \otimes_E V_k\cong \oplus_{i=0}^k t^i \Delta$. 	
	
(2)	If $\alpha \notin \Z_{<0}$, $(\Delta \otimes_E V_k)[\fc=(\alpha+k)^2-1]\cong \Delta $.

(3) If $\alpha=0$, $(\Delta \otimes_E V_k)\{\fc=k^2-1\}\cong \Delta \oplus t^k \Delta$. 
\end{proposition}
\begin{proof}
(1)	 We consider the case where $\Delta$ is equipped with the above $\gl_2$-action with $\alpha \notin \Z$.  By similar argument in the proof of Lemma \ref{keylm} and induction, we have a similar decomposition as in (\ref{decom2}) for $\Delta$ (which holds with $V_1^{\otimes k}$ replaced by $V_k$). By considering the Sen weights, the Sen weights of $(\Delta \otimes_E V_k)[(\alpha+k-2i)^2-1]$ has to be the constant $i$. Similarly as in Lemma \ref{inj}, $(\Delta \otimes_E V_k)[(\alpha+k-2i)^2-1]$ is a submodule of $\Delta$, hence is isomorphic to $t^i  \Delta$.
	
(2) (3) follows by similar argument as for Lemma \ref{keylm} (3) and Proposition \ref{transProp} (4). 
\end{proof}
\begin{remark}
	For a general de Rham $(\varphi, \Gamma)$-module  $D$, let $\DF:=D_{\pst}(D)$ (ignoring the Hodge filtration) be the Deligne-Fontaine module  associated to $D$.  By Proposition \ref{decork1} (1), the Deligne-Fontaine module associated to $D \otimes_EV_k$ is isomorphic to $\DF^{\oplus k+1}$.
\end{remark}
\section{Locally analytic representations of $\GL_2(\Q_p)$}
We show the compatibility of the translations on $(\varphi, \Gamma)$-modules and the translations on $\GL_2(\Q_p)$-representations.
\subsection{Translations of $\GL_2(\Q_p)$-sheaves on $\bP^1(\Q_p)$}
Let $\sF$ be a $\GL_2(\Q_p)$-sheaf of analytic type over $\bP^1(\Q_p)$ (cf. \cite[\S~1.3.1]{Colm16}).  For $k\geq 1$, the following data defines a $\GL_2(\Q_p)$-sheaf, denoted by $\sF \otimes_E V_k$, of analytic type over $\bP^1(\Q_p)$: for compact opens $U$, $V$ of $\bP^1(\Q_p)$, 
\begin{itemize}
	\item $(\sF \otimes_E V_k)(U)=\sF(U) \otimes_E V_k$,
\item $\Res^V_U|_{\sF\otimes_E V_k} \cong \Res^V_U|_{\sF} \otimes \id$,
\item $g_U|_{(\sF \otimes_E V_k)(U)}=g_U|_{\sF(U)} \otimes g$ for $g\in \GL_2(\Q_p)$.
\end{itemize}
Note that $w=\begin{pmatrix} 0 & 1 \\ 1 & 0 \end{pmatrix}$ induces an isomorphism $M_{\sF}^+:=\sF(\Z_p) \xrightarrow{\sim} \sF(\bP^1(\Q_p) \setminus \Z_p)$. It is clear that $M_{\sF}^+$ gives rise to a $P^+$-sheaf of analytic type over $\Z_p$, and $w$ induces an involution on $\Res_{\Z_p^{\times}}(M_{\sF}^+)$. We have $\sF(\bP^1(\Q_p))=\{(x,y)\in M_{\sF}^+ \times M_{\sF}^+\ |\ w(\Res_{\Z_p^{\times}}(x))=\Res_{\Z_p^{\times}}(y)\}$. We refer to \cite[\S~3.1.1]{Colm16} for more discussion on the relation of $\GL_2(\Q_p)$-sheaves and $P^+$-sheaves. Finally remark that the involution $w$ on $\Res_{\Z_p^{\times}}(M_{\sF \otimes_E V_k}^+)\cong \Res_{\Z_p^{\times}}(M_{\sF}^+) \otimes_E V_k$ is given by the diagonal action of $w$.

For $\mu\in E$, define $\sF[\fc=\mu]$ (resp. $\sF\{\fc=\mu\}$)  to be the subsheaf of ($\fc=\mu$)-eigenspace (resp. generalized ($\fc=\mu$)-eigenspace). It is clear that  these are 
$\GL_2(\Q_p)$-subsheaves of $\sF$ over $\bP_1(\Q_p)$.

Let $D$ be a  $(\varphi, \Gamma)$-module over $\cR_E$. Assume there is an involution $w$  on $D^{\psi=0}=D \boxtimes \Z_p^*$. Let $\delta: \Q_p^{\times} \ra E^{\times}$ be a continuous character. Assume that $(D, \delta, w)$ is compatible in the sense of \cite[\S~3.1.2]{Colm16}. Let $\sG_{D, \delta,w}$ be the associated $\GL_2(\Q_p)$-sheaf of analytic type over $\bP^1(\Q_p)$. We will frequently use Colmez's notation $D \boxtimes_{\delta,w} U:=\sG_{D,\delta,w}(U)$ for $U\subset \bP^1(\Q_p)$.

Let $w_k:=w \otimes \begin{pmatrix} 0 &1 \\ 1 &0
\end{pmatrix}$, which is an involution on $(D\otimes_EV_k) \boxtimes \Z_p^{\times} \cong (D\boxtimes \Z_p^{\times}) \otimes_E V_k$. 
\begin{proposition}\label{tranSh}If $(D, \delta, w)$ is compatible, then
	 $(D\otimes_E V_k, z^k \delta, w_k$) is compatible, and there is a natural isomorphism of $\GL_2(\Q_p)$-sheaves over $\bP^1(\Q_p)$:
	 \begin{equation*}
	 	 \sG_{D\otimes_E V_k, \delta z^k,w_k} \xrightarrow{\sim} \sG_{D,\delta,w} \otimes_E V_k.
	 \end{equation*}
 In particular, we have a $\GL_2(\Q_p)$-equivariant isomorphism $(D\otimes_E V_k) \boxtimes_{z^k \delta, w_k} \bP^1(\Q_p) \cong (D\boxtimes_{\delta, w} \bP^1(\Q_p)) \otimes_E V_k$.
\end{proposition}
\begin{proof}
	From the data $(D\otimes_E V_k, z^k \delta, w_k)$, we can construct a sheaf $\sG'$ over $\bP^1(\Q_p)$ as in \cite[\S~3.1.1]{Colm16} with $\sG'(\bP^1(\Q_p))=\{(x,y) \in (D\otimes_E V_k) \times (D\otimes_EV_k)  \ |\ w_k(\Res_{\Z_p^{\times}}(x))=\Res_{\Z_p^{\times}}(y)\}$, which is equipped with an action of the group $\tilde{G}$ in \cite[Rem. 3.1]{Colm16} using the formulas in \cite[\S~3.1.1]{Colm16}. It is then straightforward to check $(\sG_{D,\delta} \otimes_E V_k) (\bP^1(\Q_p)) \ra  \sG'(\bP^1(\Q_p))$, $(x,y) \mapsto (x,y)$ induces an isomorphism of sheaves over $\bP^1(\Q_p)$, which is equivariant under the $\tilde{G}$-action. As $\sG_{D,\delta} \otimes_E V_k$ is a $\GL_2(\Q_p)$-sheaf, the $\tilde{G}$-action on $\sG'$ factors through $\GL_2(\Q_p)$. The proposition follows.
\end{proof}

\begin{corollary}\label{faistran}Suppose $(D, \delta, w)$ is compatible. 
Let $\mu\in E$ such that $(D\otimes_E V_k)[\fc=\mu]\neq 0$. Then $((D\otimes_E V_k)[\fc=\mu], z^k \delta, w_k)$ and $((D\otimes_E V_k)\{\fc=\mu\}, z^k \delta, w_k)$ are compatible. And we have natural isomorphisms of $\GL_2(\Q_p)$-sheaves over $\bP^1(\Q_p)$:
\begin{equation*}
\sG_{D\otimes_E V_k[\fc=\mu], \delta z^k,w_k} \xrightarrow{\sim} (\sG_{D,\delta,w} \otimes_E V_k)[\fc=\mu], \ \ \sG_{D\otimes_E V_k\{\fc=\mu\}, \delta z^k, w_k} \xrightarrow{\sim} (\sG_{D,\delta, w} \otimes_E V_k)\{\fc=\mu\}.
\end{equation*}
In particular, we have $\GL_2(\Q_p)$-equivariant isomorphisms
\begin{eqnarray*}
&&	(D\otimes_E V_k)[\fc=\mu] \boxtimes_{z^k \delta, w_k} \bP^1(\Q_p) \cong (D \boxtimes_{\delta, w} \bP^1(\Q_p))[\fc=\mu], \\ && (D\otimes_E V_k)\{\fc=\mu\} \boxtimes_{z^k \delta, w_k} \bP^1(\Q_p) \cong (D \boxtimes_{\delta, w} \bP^1(\Q_p))\{\fc=\mu\}.
\end{eqnarray*}
\end{corollary}
\begin{proof}
	By the above proposition, the involution $w_k$ comes from the $w$-action on $ \sG_{D\otimes_E V_k, \delta z^k,w_k}(\Z_p^{\times})$ hence commutes with $\fc$. We see in particular $w_k$ stabilizes  $(D\otimes_E V_k)[\fc=\mu] \boxtimes \Z_p^{\times}$ and $(D\otimes_EV_k)\{\fc=\mu\} \boxtimes \Z_p^{\times}$.  The restriction maps also commute with $\fc$, hence $(D\otimes_E V_k)[\fc=\mu]=(\sG_{D\otimes_E V_k, \delta z^k,w_k}[\fc=\mu])(\Z_p)$ (resp. $(D\otimes_E V_k)\{\fc=\mu\}=(\sG_{D\otimes_E V_k, \delta z^k,w_k}\{\fc=\mu\})(\Z_p)$). The corollary then follows by the same argument as in Proposition \ref{tranSh}.
\end{proof}

\subsection{$p$-adic local Langlands correspondence for $\GL_2(\Q_p)$}
Let $\delta_D:\Q_p^{\times} \ra E^{\times}$ be the character such that $\wedge^2 D\cong \cR_E(\delta_D\varepsilon)$. Recall that by \cite[Thm.~0.1]{Colm16}, if $D$ is indecomposable, there exists a unique involution $w_D$ such that $(D,\delta_D, w_D)$ is compatible. We briefly recall the construction and some properties of $w_D$.  

(1) When $D$ is irreducible, then there exist a continuous character $\chi$ and an \'etale $(\varphi, \Gamma)$-module $D_0$ such that $D \cong D_0 \otimes_{\cR_E} \cR_E(\chi)$. Let $\cD_0$ be the continuous \'etale $(\varphi, \Gamma)$-module over $B_{\Q_p} \otimes_{\Q_p} E$ associated to $D_0$. One defines first an involution $w_{\cD_0}$ on $\cD_0^{\psi=0}$ (see \cite[Rem.~II.1.3]{Colm10a}). Then the restriction of $w_{\cD_0}$ on $(\cD_0^{\dagger})^{\psi=0}$ extends uniquely to an involution $w_{D_0}$ on $D_0$ such that  $(D_0, \delta_{D_0}, w_{D_0})$ is compatible.(cf. \cite[\S~V.2]{Colm10a}). Let $w_D:=w_{D_0} \otimes \chi(-1)$. It is straightforward to check that  $(D, \delta_D, w_D)$ is also compatible and $D \boxtimes_{\delta_D, w_D} \bP^1(\Q_p)\cong (D_0 \boxtimes_{\delta_{D_0}, w_{D_0}} \bP^1(\Q_p))\otimes_E \chi\circ \dett$.

(2) When $D$ is a non-split extension of $\cR_E(\delta_2)$ by $\cR_E(\delta_1)$. On each $\cR_E(\delta_i)^{\psi=0}$, there is a unique involution $w_{i}$ such that $(\cR_E(\delta_i), \delta_i, w_{i})$ is compatible (cf. \cite[Rem.~3.8 (i), \S~4.3]{Colm16}), and there is an exact sequence (cf. \cite[Thm.~6.8]{Colm16}):
\begin{equation*}
0 \ra \cR_E(\delta_1) \boxtimes_{\delta_D, w_1} \bP^1(\Q_p) \ra D \boxtimes_{\delta_D, w_D} \bP^1(\Q_p) \ra \cR_E(\delta_2) \boxtimes_{\delta_D, w_2} \bP^1(\Q_p) \ra 0.
\end{equation*}

\begin{remark}
	
	(1) If $D$ contains a pathological submodule, i.e. up to twist $D$ is isomorphic to a non-de Rham extension $\cR_E- t^n \cR_E$ with $n\in \Z_{\geq 0}$, then by \cite[\S~6.5.1, 6.5.2]{Colm16}, the $\fc$-action on $D\boxtimes_{\delta_D, w_D} \bP^1(\Q_p)$ is not scalar. While for other cases, $\fc$ is scalar. 
	
	(2) Suppose $D$ does not have pathological submodules, and assume $D$ has Sen weights $(0,\alpha_D)$ with $\alpha_D\in E\setminus \Z_{<0}$. The $\gl_2$-action on $D$ induced from $D \boxtimes_{\delta_D, w_D} \bP^1(\Q_p)$ coincides with the one given in \S~\ref{sec22} with $\alpha=0$ when $\alpha_D=0$. 
\end{remark}

We write $D \boxtimes_{\delta_D} \bP^1(\Q_p):=D \boxtimes_{\delta_D, w_D} \bP^1(\Q_p)$. Recall that we have a $\GL_2(\Q_p)$-equivariant exact sequence (cf. \cite[Thm.~0.1]{Colm16})
\begin{equation}\label{pLL1}
	0 \ra \pi(D)^* \otimes_E \delta_D \circ \dett \ra D \boxtimes_{\delta_D} \bP^1(\Q_p) \ra \pi(D) \ra 0,
\end{equation}
where $\pi(D)$ is the locally analytic representation of $\GL_2(\Q_p)$ (of central character $\delta_D$) corresponding to $D$ in the $p$-adic local Langlands correspondence. Note that if $D'\cong D \otimes_{\cR_E}\cR_E(\chi)$, then $D'\boxtimes_{\delta_{D'}} \bP^1(\Q_p) \cong (D \boxtimes_{\delta_D} \bP^1(\Q_p) ) \otimes_E \chi \circ \dett$, hence  $\pi(D')\cong \pi(D) \otimes_E \chi \circ \dett$. 

\subsection{Change of weights}
Twisting $D$ by a continuous character, we assume $D$ has Sen weights $(0,\alpha_D)$ with $\alpha_D\in E \setminus \Z_{<0}$. Let $k\in \Z_{\geq 1}$. 
We deduce from (\ref{pLL1}) an exact sequence
\begin{equation*}
	0 \ra \pi(D)^* \otimes_E V_k \otimes_E \delta_D \circ \dett \ra (D \otimes_E V_k) \boxtimes_{\delta _Dz^k, w_{D,k}} \bP^1 \ra \pi(D) \otimes_E V_k \ra 0.
\end{equation*}
Let $\mu\in E$, we have two exact sequences:
\begin{equation*}
	0 \ra (\pi(D)^* \otimes_E V_k \otimes_E \delta_D)\{\fc=\mu\} \ra (D \otimes_E V_k)\{\fc=\mu\} \boxtimes_{\delta _Dz^k, w_{D,k}}  \bP^1 \ra (\pi(D) \otimes_E V_k)\{\fc=\mu\} \ra 0,
\end{equation*}
\begin{equation}\label{seqinf}
	0 \ra (\pi(D)^* \otimes_E V_k \otimes_E\delta_D)[\fc=\mu] \ra (D \otimes_E V_k)[\fc=\mu] \boxtimes_{\delta _Dz^k, w_{D,k}}  \bP^1 \ra (\pi(D) \otimes_E V_k)[\fc=\mu].
\end{equation}
\begin{theorem}\label{CoWthm1}Assume $D$ is indecomposible and $D$ does not have pathological submodules.

(1)	Assume $\alpha_D\notin \Z$ or $\alpha_D\in \Z$ and $\alpha\geq k$. For $i=0, \cdots, k$, $D_{(i,\alpha+k-i)} \boxtimes_{\delta _{D_{(i,\alpha+k-i)}}} \bP^1(\Q_p) \cong (D \boxtimes_{\delta_D} \bP^1(\Q_p))[\fc=(\alpha+k-2i)^2-1]$ and $\pi(D_{(i,\alpha+k-i)}) \cong (\pi(D) \otimes_E V_k)[\fc=(\alpha+k-2i)^2-1]$.

(2) Assume $\alpha_D\neq 0$ or $D$ not de Rham, then 
$D_{(0,\alpha+k)}\boxtimes_{\delta_{D_{(0,\alpha+k)}}} \bP^1(\Q_p) \cong (D \boxtimes_{\delta_D} \bP^1(\Q_p)[\fc=(\alpha+k)^2-1]$ and $\pi(D_{(0,\alpha+k)})\cong (\pi(D) \otimes_E V_k)[\fc=(\alpha+k)^2-1]$.
\end{theorem}

\begin{proof}
	The first isomorphisms in (1) and (2) follow directly from Corollary \ref{faistran}, Proposition \ref{transProp} and the uniqueness of the compatible involution (cf. \cite[Prop.~3.17, Rem.~3.8]{Colm16}, noting $\delta_{D_{(i,\alpha+k-i)}}=z^k \delta_D$). For the second isomorphisms, we only prove (1), (2) following by similar arguments. We have an exact sequence
	\begin{equation*}
	0 \ra \pi(D_{(i,\alpha+k-i)})^* \otimes_E \delta_{D_{(i,\alpha+k-i)}} \ra D_{(i,\alpha+k-i)} \boxtimes_{\delta_{D_{(i,\alpha+k-i)}}} \bP^1(\Q_p) \ra \pi(D_{(i,\alpha+k-i)}) \ra 0.
	\end{equation*}
	By (\ref{pLL1}), the same argument as in \cite[Lem.~3.21]{Colm18} and the fact $\pi(D_{(i,\alpha+k-i)})$ does not have finite dimensional subrepresentations, we see the injection $$(\pi(D)^* \otimes_E V_k \otimes_E\delta_D)[\fc=(\alpha+k-2i)^2-1] \hooklongrightarrow D_{(i,\alpha+k-i)} \boxtimes_{\delta_{D_{(i,\alpha+k-i)}}} \bP^1(\Q_p)$$ factors through $\pi(D_{(i,\alpha+k-i)})^* \otimes_E \delta_{D_{(i,\alpha+k-i)}}$. The quotient of $\pi(D_{(i,\alpha+k-i)})^* \otimes_E \delta_{D_{(i,\alpha+k-i)}}$ by $(\pi(D)^* \otimes_E V_k \otimes_E\delta_D)[\fc=(\alpha+k-2i)^2-1]$ injects into the $E$-space of compact type $(\pi(D) \otimes_EV_k)[\fc=(\alpha+k-2i)^2-1]$, which, by the same argument as in \cite[Lem.~3.21]{Colm18}, has to be finite dimensional. As $\pi(D_{(i,\alpha+k-i)})$ does not have finite dimensional subrepresentations, we deduce the second isomorphism in (1). 
\end{proof}
\begin{remark}\label{rem35}
	(1) When $D$ is trianguline, certain cases (concerning $\pi(D)$) were also obtained in \cite[Thm.~5.2.11]{JLS21}.
	

(2) Suppose $\alpha_D=0$ and $D$ is not de Rham.  By Theorem \ref{CoWthm1} (2), one easily sees the right map in (\ref{seqinf})  for such $D$ and $\mu=k^2-1$ is surjective.

\end{remark}
We move to $\alpha=0$ and  de Rham case. This case is different, as the translation in this case does not directly  give non-trivial $(\varphi, \Gamma)$-submodules (i.e. submodules other than $t^i D$). If $D$ is moreover non-trianguline, we let $\pi(D,i)$ for $i\in \Z$ be Colmez's representations in \cite{Colm18} (for $D=\nabla$ of \textit{loc. cit.}). 
\begin{theorem}\label{CoWThm2}
	Assume $D$ is indecomposable, de Rham of Hodge-Tate weights $(0,0)$. Then $(D, z^k \delta_D, w_{D,k})$  and $(t^k D, z^k \delta_D, w_{D,k})$ are  compatible. We have $$D \boxtimes_{z^k \delta_D, w_{D,k}} \bP^1(\Q_p)\cong (D\boxtimes_{\delta_D} \bP^1(\Q_p))[\fc=k^2-1],$$ and an exact sequence
	\begin{equation*}
	0 \ra D \boxtimes_{z^k \delta_D, w_{D,k}} \bP^1(\Q_p)\ra 	((D\boxtimes_{\delta_D} \bP^1(\Q_p)) \otimes_E V_k)\{\fc=k^2-1\} \ra t^kD \boxtimes_{z^k \delta_D, w_{D,k}} \bP^1(\Q_p)  \ra 0.
	\end{equation*}
 If $D$ is moreover non-trianguline, then $(\pi(D)^* \otimes_EV_k)[\fc=k^2-1]\cong \pi(D,k)^*$, and we have an exact sequence
	\begin{equation}\label{genec1}
		0 \ra \pi(D,k)^* \ra (\pi(D)^* \otimes_E  V_k)\{\fc=k^2-1\} \ra \pi(D,-k)^* \ra 0.
		\end{equation}
\end{theorem}

\begin{proof}
The first part of the theorem follows directly from Proposition \ref{transProp} (4) and Corollary \ref{faistran}. Assume $D$ is non-trianguline, by the uniqueness of the involution (\cite[Prop.~3.17]{Colm16}) $D \boxtimes_{z^k \delta_D, w_{D,k}} \bP^1(\Q_p)$ (resp. $t^k D \boxtimes_{z^k \delta_D, w_{D,k}} \bP^1(\Q_p)$) is just the representation $D \boxtimes_{z^k \delta_D} \bP^1(\Q_p)$ (resp. $t^k D \boxtimes_{z^k \delta_D} \bP^1(\Q_p)$) of \cite[\S~3.3]{Colm18}.  Similarly as in the proof of Theorem \ref{CoWthm1}, using the same argument as in \cite[Lem.~3.21]{Colm18} by comparing (\ref{seqinf}) and the exact sequence in \cite[Rem.~3.20]{Colm18}, we deduce $(\pi(D)^* \otimes_EV_k)[\fc=k^2-1]\cong \pi(D,k)^*$. By \cite[Prop.~3.23]{Colm18}, we have
\begin{equation*}
0 \ra \pi(D,-k)^* \otimes_E \delta_D \ra t^kD \boxtimes_{z^k \delta_D} \bP^1(\Q_p) \ra \pi(D,k) \ra 0.
\end{equation*}
Again by similar arguments in \cite[Lem.~3.21]{Colm18}, the composition 
\begin{equation*}(\pi(D)^* \otimes_E V_k \otimes_E\delta_D)\{\fc=k^2-1\} \ra ((D \boxtimes_{\delta_D} \bP^1(\Q_p))\otimes_E V_k)\{\fc=k^2-1\} \twoheadrightarrow t^k D \boxtimes_{z^k \delta_D} \bP^1(\Q_p)
\end{equation*} factors through $\kappa: (\pi(D)^* \otimes_E V_k)\{\fc=k^2-1\}  \ra \pi(D,-k)^*$. Similarly, the composition  
\begin{multline}\pi(D,-k)^* \otimes_E\delta_D \lra   ((D \boxtimes_{\delta_D} \bP^1(\Q_p))\otimes_E V_k)\{\fc=k^2-1\}/D \boxtimes_{z^k \delta_D} \bP^1(\Q_p)\\ \lra (\pi(D) \otimes_E V_k)\{\fc=k^2-1\}/\pi(D,-k)
  	\end{multline} has to be zero, so $\kappa$ is surjective. This concludes the proof. 
\end{proof}
\begin{remark}\label{rem36}
(1) We have $t^k D \boxtimes_{z^k \delta_D, w_{D,k}} \bP^1(\Q_p) \cong (D \boxtimes_{z^{-k} \delta_D} \bP^1(\Q_p)) \otimes_E z^k \circ \dett\cong (\check{D} \boxtimes_{z^k \delta_{\check{D}}} \bP^1(\Q_p))^{\vee} \otimes_E z^k \circ \dett$, where $\check{D}:=D^{\vee} \otimes_E \varepsilon$, see \cite[Prop.~3.2]{Colm16} for the last isomorphism.

(2) As $\pi(D,-k)^* \subsetneq \pi(D,k)^*$, we see that the right map in (\ref{seqinf}) is not surjective in the case (for $\mu=k^2-1$), in contrast to Remark \ref{rem35} (2).

(3) The second part of the theorem also holds in the trianguline case. We discuss the representations $\pi(D,i)$ in the corresponding cases. Twisting by smooth characters, there are only two cases (noting $D$ is indecomposable):
\begin{enumerate}
	\item[(a)]$D\cong [\cR_E(|\cdot|) - \cR_E]$,
	\item[(b)]$D \cong [\cR_E-\cR_E]$, is the unique de Rham non-split extension.
	\end{enumerate} 
For the case (a), we let $\pi(D,-i):=\Pi_i$, $\pi(D,i):=\Pi_i'$ be as in \cite[Prop.~6.13]{Colm16} for $i\in \Z_{>0}$. The second part of Theorem \ref{CoWThm2} for such $D$ follows by similar arguments in the proof and \cite[Prop.~6.13]{Colm16}.

\noindent For the case (b), let $\val_p:=\Q_p^* \ra E$ be the smooth character sending $p$ to $1$, to which we associate a smooth character $\eta: \Q_p^* \ra E[\epsilon]/\epsilon^2$, $a \mapsto 1+\val_p(a)\epsilon$. Note $\eta$ is a two dimensional smooth representation of  $\Q_p^*$ over $E$.  For $i\in \Z_{\geq 0}$ (resp. $i\in \Z_{<0}$), let $\pi(D,i):=(\Ind_{B^-(\Q_p)}^{\GL_2(\Q_p)} \varepsilon^{-1}z^i \eta \otimes 1)^{\an}$ (resp. $\pi(D,i):=(\Ind_{B^-(\Q_p)}^{\GL_2(\Q_p)} \varepsilon^{-1} z^i \eta \otimes 1)^{\an} \otimes_E z^{-i} \circ \dett$) (which has central character $\varepsilon^{-1}z^{|i|}$). By discussions in \cite[\S~6.5.1]{Colm16}, the second part of Theorem \ref{CoWThm2} in this case similarly follows. 

(4) In all cases, let $\pi_{\infty}(D)$ be the smooth representation of $\GL_2(\Q_p)$ associated to $D$ via the classical local Langlands correspondence. Let $D'\subset D$ be a $(\varphi, \Gamma)$-submodule of Sen weights $(0,k)$, and assume $D'$ is indecomposable. By \cite[Prop.~2.4, Rem.~2.5]{Colm18}, we have (note $\delta_{D'}=z^k \delta_D$)
 \begin{equation*}
 	D' \boxtimes_{\delta_{D'}} \bP^1(\Q_p) \hooklongrightarrow D \boxtimes_{z^k \delta_D, w_{D,k}} \bP^1(\Q_p),
 \end{equation*}
which induces $\pi(D')^* \hookrightarrow \pi(D,k)^*$. Moreover, we have 
\begin{equation*}
	\pi(D,-k)^* \subsetneq \pi(D')^* \subsetneq \pi(D,k)^*,
\end{equation*}
with $\pi(D,k)^*/\pi(D')^*\cong \pi(D')^*/\pi(D,-k)^* \cong (\pi_{\infty}(D) \otimes_E V_k)^*$. When $D$ is irreducible or be as in case (3)(a), $\pi(D,k)^*/\pi(D,-k)^*\cong ((\pi_{\infty}(D) \otimes_E V_k)^*)^{\oplus 2}$. Moreover, by \cite[Thm.~6.15]{Colm16} \cite[Thm.~0.6(iii)]{Colm18}, the map $D' \mapsto \pi(D')^*$ is a one-to-one correspondence between the $(\varphi, \Gamma)$-submodules of Sen weights $(0,k)$ to the subrepresentation of $\pi(D,k)^*$ of quotient isomorphic to $(\pi_{\infty}(D) \otimes_E V_k)^*$ (which also corresponds to non-split extensions of $(\pi_{\infty}(D) \otimes_E V_k)^*$ by $\pi(D,-k)^*$, cf. \cite[Lem.~3.1.3]{Br16}\cite[Thm.~2.5]{Ding10}). When $D$ is as in (3)(b), then $D$ admits a unique indecomposible $(\varphi, \Gamma)$-submodule of weights $(0,k)$, which has the form $[t^k \cR_E -\cR_E]$. Correspondingly, in this case $\pi(D,k)^*/\pi(D,-k)^*$ is a non-split self-extension of $(\pi_{\infty}(D) \otimes_E V_k)^*$, and $\pi(D')^*$ is the (unique) subrepresentation of $\pi(D,k)^*$ of quotient $(\pi_{\infty}(D) \otimes_E V_k)^*$.
\end{remark}
We finally discuss the translations on the $\GL_2(\Q_p)$-sheaves associated to rank one $(\varphi, \Gamma)$-modules. Twisting by characters, it suffices to consider $\cR_E \boxtimes_{\delta} \bP^1(\Q_p)$ for a continuous character $\delta$ of $\Q_p^*$.  Let $\alpha:=\wt(\delta)+1$, the corresponding $\gl_2$-action on $\cR_E$ satisfies $\fz=\wt(\delta)$, and $u^-=\frac{\nabla(\nabla-\alpha)}{t}$. We have by \cite[Prop.~4.14]{Colm16},  $(\cR_E \boxtimes_{\delta} \bP^1(\Q_p))^{\vee} \cong \cR_E(\varepsilon) \boxtimes_{\delta^{-1}} \bP^1(\Q_p)\cong (\cR_E \boxtimes_{\delta^{-1} \varepsilon^{-2}} \bP^1(\Q_p)) \otimes_E \varepsilon\circ \dett$. So it suffices to consider the case $\alpha \in E\setminus \Z_{<0}$, and we assume this is so. The following theorem follows easily from Proposition \ref{decork1} (applied to $\Delta=\cR_E$) and Corollary \ref{faistran}.
\begin{theorem}
	We have
$((\cR_E \boxtimes_{\delta} \bP^1(\Q_p)) \otimes_EV_k)[\fc=(\alpha+k)^2-1]\cong \cR_E \boxtimes_{z^k \delta} \bP^1(\Q_p)$, and an exact sequence
\begin{equation*}
	0 \ra \cR_E \boxtimes_{z^k \delta} \bP^1(\Q_p)  \ra ((\cR_E \boxtimes_{\delta} \bP^1(\Q_p)) \otimes_EV_k)\{\fc=(\alpha+k)^2-1\} \ra t^k \cR_E  \boxtimes_{z^k \delta} \bP^1(\Q_p) \ra 0.
\end{equation*}\end{theorem}
\begin{remark}By \cite[Prop.~4.12]{Colm16}, we have an exact sequence
\begin{equation*}
	0 \ra ((\Ind_{B^-(\Q_p)}^{\GL_2(\Q_p)} \delta \otimes 1) ^{\an})^*\otimes_E \delta \circ \dett \ra \cR_E \boxtimes_{\delta} \bP^1(\Q_p) \ra (\Ind_{B^-(\Q_p)}^{\GL_2(\Q_p)} \varepsilon^{-1} \otimes \delta \varepsilon)^{\an} \ra 0.
\end{equation*}
We refer to \cite{JLS21} for a detailed study of translations on locally analytic principal series.
\end{remark}

 \subsection{Some complements}\label{seccomp} In this section, let $D$ be an indecomposible $(\varphi, \Gamma)$-module of rank $2$, equipped with the induced $\gl_2$-action from $D \boxtimes_{\delta_D} \bP^1(\Q_p)$. We assume moreover $D$ does not contain pathological submodules.

 Suppose $D$ is not trianguline. 
We reveal and generalize Colmez's operator $\partial$ on $\pi(D)^*$ \cite{Colm18}. By \cite[Cor.~2.7]{Dos11}, $u^+$ is injective on $\pi(D)^*$. By the same argument as in Lemma \ref{inj}, we have
\begin{lemma}
	Assume $D$ is not trianguline, the $P^+$-equivariant composition 
	\begin{equation}\label{jmathk}
	\jmath_k:	(\pi(D)^* \otimes_E V_k)[\fc=\mu] \hooklongrightarrow \pi(D)^* \otimes_E V_k \twoheadlongrightarrow \pi(D)^*
	\end{equation}
is injective. 
\end{lemma}

The following lemma follows by direct calculation.
\begin{lemma}\label{lie}
Let $M$ be	an $E$-vector space equipped with a $\gl_2$-action. Let $\alpha\in E$, assume on $M$, $\fc=\alpha^2-1$ and $\fz=\alpha-1$. Then for $g=\begin{pmatrix}
	a & b \\ c & d
\end{pmatrix}$, we have on $M$:
\begin{equation*}
	u^+ \Ad_g(u^+)=(-ca^++au^+)(-c(a^+-\alpha)+au^+).
\end{equation*}
In particular, if $u^+$, and $\Ad_g(u^+)$ are injective on $M$, so are the operators $(-ca^++au^+)$ and $(-c(a^+-\alpha)+a u^+)$. 
\end{lemma}
Consider the $k=1$ case.	By the same argument as in the proof of Lemma \ref{keylm} (with $D$ replaced by $\pi(D)^*$), $(\pi(D)^* \otimes_E V_1)[\fc=(\alpha+1)^2-1]$ has the form $v=v_0\otimes e_0+v_1 \otimes e_1$ with $v_0 \in \pi(D)^*$ satisfying $a^+(v_0)\in u^+ \pi(D)^*$ and $v_1=\frac{a^+}{u^+} v_0$ (well-defined as $u^+$ is injective). The map $\jmath_1:(\pi(D)^* \otimes_EV_1)[\fc=(\alpha+1)^2-1] \hookrightarrow \pi(D)^* \otimes_E V_1 \ra \pi(D)^*$ sends $v$ to $v_0$.  We let $\partial:=\frac{a^+}{u^+}: \Ima(\jmath_1) \ra \pi(D)^*$.
\begin{lemma}\label{lemprt}
 For $g=\begin{pmatrix}
	a & b \\ c & d
	\end{pmatrix}\in \GL_2(\Q_p)$, and $u\in \Ima(\jmath_1)$, we have $g(u)\in (-c\partial+a) \Ima(\jmath_1)$. Moreover, $\jmath_1(g(v))=\det(g) (-c\partial +a)^{-1} g(\jmath_1(v))$ for $v\in (\pi(D)^* \otimes_EV_1)[\fc=(\alpha+1)^2-1]$.
\end{lemma}
\begin{proof}
Write $v=v_0\otimes e_0+v_1 \otimes e_1$. Hence $\jmath_1(g(v))=c g(v_1)+dg(v_0)$. As $a^+(v_0)=u^+(v_1)$, we have $\Ad_g(a^+) g(v_0)= \Ad_g(u^+) g(v_1)$. So $u^+ \Ad_g(u^+) \jmath_1(g(v))=u^+\Ad_g(ca^++du^+) g(v_0)$. By a direct calculation, $\Ad_g(ca^++du^+)=\dett(g)(-c(a^+-\alpha+1)+au^+)$. Together with Lemma \ref{lie}, $(-ca^++au^+) \jmath_1(g(v))=u^+ g(\jmath(v))$. The lemma follows.
\end{proof}
Let $\jmath_1^i$ be the similar map with $(\pi(D)^*\otimes_E V_{i-1})[\fc=(\alpha+i-1)^2-1]$ replacing $\pi(D)^*$. It is easy to see by induction that (\ref{ind2}) holds with $D$ replaced by $\pi(D)^*$. We have $\jmath_k=\jmath_1^k \circ \jmath_1^{k-1}\cdots \circ \jmath_1^1$, and operators:
\begin{equation}\label{partialk}
\Ima(\jmath_1^{k}) \xlongrightarrow{\partial} \Ima(\jmath_1^{k-1}) \xlongrightarrow{\partial} \cdots \xlongrightarrow{\partial} \Ima(\jmath_1^2) \xlongrightarrow{\partial} \pi(D)^*.
\end{equation}
By Lemma \ref{lemprt} and induction (with an analogue of (\ref{ind2})), we have:
\begin{proposition}For $g=\begin{pmatrix} a & b \\ c & d \end{pmatrix}\in \GL_2(\Q_p)$, and $u\in \Ima(\jmath_k)$, we have $g(u) \in (-c\partial+a)^k \Ima(\jmath_k)$. Moreover, $\jmath_k(g(v))=\det(g)^k (-c\partial +a )^{-k} g(\jmath_k(v))$.
\end{proposition}
\begin{remark}
(1)	In particular, one can construct the representation $(\pi(D)^* \otimes_EV_k)[\fc=(\alpha+k)^2-1]$ from $\pi(D)^*$: Let $M$ be the subspace of $\pi(D)^*$ consisting of vectors $v$ such that $\nabla_i(v)\in (u^+)^i \pi(D)^*$ for $i=1,\cdots, k-1$, where $\nabla_i:=(a^+-i+1)\cdots a^+$. For $g\in \GL_2(\Q_p)$ and $v\in M$, one can show that $g(v)$ lies in  $(-c\partial+a)^k M$. The formula $$g *_k v:=\det(g)^k (-c\partial+a)^k g(v)$$
	defines a $\GL_2(\Q_p)$-action on $M$. The topology on $M$ is a bit subtle. If $M$ is closed in $\pi(D)^*$ (for example when $D$ is de Rham, by \cite[Prop.~9.1]{DLB}), we equip $M$ with the induced topology. In general, using (\ref{casmirk}), from $v_0:=v\in M$, we inductively construct $\{v_i\}_{i=0,\cdots, k}$ with $v_i \in \pi(D)^*$, and obtain an injection $M \hookrightarrow \pi(D)^* \otimes_E V_k$, $v\mapsto \sum_{i=0}^{k} v_i \otimes e_i$. The image is closed with $\pi(D)^* \otimes_E V_k\cong (\pi(D)^*)^{\oplus k+1}$ as topological vector space (as it is exactly the $\fc=(\alpha+k)^2-1$ eigenspace), and we equip $M$ with the induced topology. It is then clear  $M\cong (\pi(D)^* \otimes_E V_k)[\fc=(\alpha+k)^2-1]$. When $D$ is de Rham of constant Sen weights $(0,0)$, $\Ima(\jmath_k)=D$, this reveals Colmez's construction of $\pi(D,k)^*$ ($\cong (\pi(D)^* \otimes_EV_k)[\fc=k^2-1]$).

	(2) If $u^+$ is not injective or equivalently $D$ is trianguline, the kernel of   $\jmath_1$ consists  exactly of $v_1 \otimes e_1$ with $v_1\in \pi(D)^*[u^+=0]$, and is not stabilized by $\GL_2(\Q_p)$. So in this case, we can not directly construct $(\pi(D)^* \otimes_E V_1)[\fc=(\alpha+1)^2-1]$ from certain subspaces of $\pi(D)^*$ using a twisted $\GL_2(\Q_p)$-action. 
\end{remark}
Finally, we discuss the relation of involutions. We keep the assumption on $D$ in the first paragraph of the section (while $D$ can be trianguline). Let $D'\subset D$ be a $(\varphi, \Gamma)$-submodule of Sen weights $(0,\alpha+k)$. If $\alpha\neq 0$ or $\alpha=0$ and $D$ is not de Rham, then $D'\cong D_{(0,\alpha+k)}$. Similarly, as in (\ref{partialk}), we have operators
\begin{equation*}
\partial^k: D_{(0,\alpha+k)} \xlongrightarrow{\partial} D_{(0,\alpha+k-1)} \xlongrightarrow{\partial} \cdots \xlongrightarrow{\partial} D_{(0,\alpha+1)} \xlongrightarrow{\partial} D.
\end{equation*}
If $\alpha=0$ and $D$ is de Rham, then similarly as in (\ref{partialk}) we have $\partial^k: D\ra D$. In any case, we have $\partial^k=\frac{\nabla_k}{t^k}$. We have the following relation on the involutions.
\begin{proposition} 
	We have $w_{D'}=w_D\circ \frac{\nabla_k}{t^k}=w_D \circ \partial^k$.
\end{proposition}
\begin{proof}
	We only prove the case for $k=1$, the general case following by an induction argument.
	
	 Assume first $\alpha\neq 0$ or $\alpha=0$ and $D$ is not de Rham. We have $D'=D_{(0,\alpha+1)}$. By Theorem \ref{CoWthm1}, we see $w_{D'}=w_D \otimes \begin{pmatrix}
	0 & 1 \\ 1 & 0
	\end{pmatrix}$ with $(D')^{\psi=0}\hookrightarrow D^{\psi=0} \otimes_E V_1$. For $v=v_0 \otimes e_0+v_1 \otimes e_1$, $\jmath_1(w_{D'}(v))=w_D(v_1)=w_D(\partial(v_0))=w_D(\partial(\jmath_1(v)))$. 
	
	Assume now $\alpha=0$ and $D$ de Rham, by the same argument we have $w_{D,1}=w_D \circ \partial$ as operator on $D^{\psi=0}$. By \cite[Prop.~2.4, Rem.~2.5]{Colm18}, $w_{D'}=w_{D,1}|_{(D')^{\psi=0}}$. The proposition follows. 
\end{proof}

\end{document}